\documentclass[11pt]{article}
\usepackage[T1]{fontenc}
\usepackage{amsthm, amsmath, amssymb, amsfonts, url, booktabs, tikz, setspace, fancyhdr, bm}
\usepackage{hyperref}
\usepackage{geometry}
\usepackage{enumerate}
\usepackage[shortlabels]{enumitem}
\usepackage[babel]{microtype}
\usepackage[english]{babel}
\usepackage[capitalise]{cleveref}
\usepackage{cite}
\usepackage{comment}
\usepackage{bbm}
\usepackage{csquotes}
\usepackage{mathabx}
\usepackage{tikz}
\usepackage{graphicx}
\usepackage{float}
\usepackage{diagbox}
\counterwithin{figure}{section}

\newtheorem{theorem}{Theorem}[section]

\newtheorem{lemma}[theorem]{Lemma}

\newtheorem{cor}[theorem]{Corollary}
\newtheorem{conj}[theorem]{Conjecture}

\newtheorem*{thm}{Theorem}

\theoremstyle{definition}

\newtheorem*{defn-non}{Definition}

\newcounter{propcounter}

\usepackage{todonotes}

\newcommand{\eps}{\varepsilon}

\newcommand{\ex}{\mathrm{ex}}

\title{Rainbow spanning structures in strongly edge-colored graphs}
\author{Laihao Ding\thanks{School of Mathematics and Statistics, and Key Laboratory of Nonlinear Analysis \& Applications (Ministry of Education), Central China Normal University,
Wuhan, China. LHD is supported by the National Natural Science Foundation of China (11901226), the Fundamental Research Funds for the Central Universities (CCNU25JCPT031). XLH is supported by the National Natural Science Foundation of China (12471325). Emails: dinglaihao@ccnu.edu.cn, xlhu@ccnu.edu.cn.}
\and Xiaolan Hu\footnotemark[1] 
\and Suyun Jiang\thanks{School of Artificial Intelligence, Jianghan University, Wuhan, Hubei, China. 
Supported by Hubei Provincial Natural Science Foundation of China (2025AFB666) and National Natural Science Foundation of China (11901246).
Email: jiang.suyun@163.com.}
}

\date{}

\begin{document}

\maketitle
\begin{abstract}
An edge-colored graph is a graph in which each edge is assigned a color.
Such a graph is called strongly edge-colored if each color class forms an induced matching, and called rainbow if all edges receive pairwise distinct colors. In this paper, by establishing a connection with $\mu n$-bounded graphs, we prove that for all sufficiently large integers $n$, every strongly edge-colored graph $G$ on $n$ vertices with minimum degree at least $\frac{n+1}{2}$ contains a rainbow Hamilton cycle. We also characterize all strongly edge-colored graphs on $n$ vertices with minimum degree exactly $\frac{n}{2}$ that do not contain a rainbow Hamilton cycle. As an application, we determine the optimal minimum degree conditions for the existence of rainbow Hamilton paths and rainbow perfect matchings in strongly edge-colored graphs. 
Together, these results verify three conjectures concerning strongly edge-colored graphs for sufficiently large $n$.

\end{abstract}

\section{Introduction}

All graphs considered in this article are simple and finite. For terminology and notation not defined here, we refer the reader to \cite{bondy}. An \textit{edge-colored graph} is a graph in which each edge is assigned a color. For an edge-colored graph $G$, we denote by $\mathcal{C}_G$ the set of all colors appearing on the edges of $G$. A \textit{color class} of $G$ is a maximal subset of $E(G)$ in which all edges have the same color. If every color class of $G$ is a matching, an induced matching, a single edge, or has at most $k$ edges, then $G$ is said to be \textit{properly edge-colored}, \textit{strongly edge-colored}, \textit{rainbow}, or \textit{$k$-bounded}, respectively. For every vertex $v\in V(G)$, the \textit{color degree} of $v$ is the number of colors appearing on the edges incident to $v$. The \textit{minimum color degree} of $G$, denoted by $\delta^c(G)$, is the minimum color degree over all vertices. Finally, we write $d_G(v)$ for the degree of a vertex $v$ in $G$, and $\delta(G)$ for the minimum degree of $G$.   

The study of finding certain structures in edge-colored graphs has a long history. By now, a wide range of topics concerning the existence of monochromatic, rainbow, or properly edge-colored subgraphs under various settings have been developed. Prominent examples include Ramsey theory, rainbow Tur\'{a}n problems \cite{rainbowcycle2,evencycle,rainbowturan,rainbowcycle1}, edge-colored Dirac-type problems \cite{boundedmat,bounded,lo}, and transversals in graph systems \cite{cheng1}.  These developments have also revealed connections to other areas of mathematics, such as quasi-random hypergraphs, digraph theory, and additive number theory.

In the past decade, finding properly edge-colored or rainbow cycles under color degree conditions has been extensively studied. 
In 2013, Li \cite{lihao} proved that every edge-colored $n$-vertex graph $G$ with $\delta^c(G)\geq\frac{n+1}{2}$ contains a rainbow triangle. Extending Li's result, Czygrinow, Molla, Nagle and Oursler \cite{oddcycle} proved that any rainbow $\ell$-cycle can be found with the same hypothesis whenever $n\geq 432\ell$. Moreover, for even $\ell$, they \cite{evencycle2} showed that $\delta^c(G)\geq\frac{n+5}{3}$ is enough to ensure a rainbow $\ell$-cycle if $n$ is sufficiently large compared to $\ell$. Ding, Hu, Wang and Yang \cite{ding} investigated the existence of properly edge-colored short cycles under color degree conditions, and established a close connection to the problem of finding directed cycles under out-degree conditions. 

Compared to short cycles, finding long properly edge-colored or rainbow cycles seems to be more difficult. A celebrated result due to Lo \cite{lo} states that any edge-colored $n$-vertex graph $G$ with $\delta^c(G)\geq\frac{2n}{3}$ contains a properly edge-colored Hamilton cycle if $n$ is sufficiently large. For the rainbow version, Maamoun and Meyniel \cite{norainbowHC} found proper edge-colorings of complete graphs $K_n$ without rainbow Hamilton path for $n=2^k$, disproving a conjecture of Hahn that every properly edge-colored complete graph admits a rainbow Hamilton path. Therefore, no rainbow Hamilton cycle is ensured even in a properly edge-colored complete graph. Motivated by this, two natural questions arise: what is the maximum length of a rainbow cycle in a properly edge-colored complete graph, and when does a rainbow Hamilton cycle exist in an edge-colored graph.

For the first question,  Alon, Pokrovskiy and Sudakov \cite{rainbowHC} proved that every properly edge-colored complete graph $K_n$ contains a rainbow cycle of length at least $n-O(n^{3/4})$. Later, Balogh and Molla \cite{rainbowHC2} improved this result by reducing the term $O(n^{3/4})$ to $O(\sqrt{n}\log n)$. For the second one, Coulson and Perarnau \cite{bounded} proved that every $\mu n$-bounded $n$-vertex graph $G$ with $\delta(G)\geq\frac{n}{2}$ has a rainbow Hamilton cycle for sufficiently large $n$ and sufficiently small $\mu$. 
Cheng, Sun, Tan and Wang \cite{cheng2} considered strongly edge-colored graphs and proved the following.

\begin{theorem}[Cheng, Sun, Tan and Wang, \cite{cheng2}]\label{2/3}
Every strongly edge-colored $n$-vertex graph $G$ with $\delta(G)\geq\frac{2n}{3}$
has a rainbow Hamilton cycle.   
\end{theorem}
 
In addition, Cheng, Sun, Tan and Wang \cite{cheng2} proposed the following two conjectures and showed that the conjectured minimum degree conditions, if valid, will be optimal.  

\begin{conj}[Cheng, Sun, Tan and Wang, \cite{cheng2}]\label{cycle}
Every strongly edge-colored $n$-vertex graph $G$ with $\delta(G)\geq\frac{n+1}{2}$
has a rainbow Hamilton cycle.
\end{conj}

\begin{conj}[Cheng, Sun, Tan and Wang, \cite{cheng2}]\label{path}
Every strongly edge-colored $n$-vertex graph $G$ with $\delta(G)\geq\frac{n}{2}$
has a rainbow Hamilton path.
\end{conj}

In this paper, we confirm the above two conjectures for sufficiently large $n$ by proving the following two stronger results, where Theorem~\ref{HP} is an immediate consequence of Theorem~\ref{HC}. Here, for any vertex set $V$, $G[V]$ denotes the induced subgraph of $G$ on $V$.

\begin{theorem}\label{HC}
There exists $n_0\in \mathbb{N}$ such that if $G$ is a strongly edge-colored $n$-vertex graph with $n\geq n_0$ and $\delta(G)\geq \frac{n}{2}$, then 
\begin{itemize}[leftmargin=2em]
    \item either $G$ contains a rainbow Hamilton cycle, 
    \item or $G$ admits a partition $V(G)=V_1\cup V_2$ with $|V_1|\leq |V_2|$ such that $G[V_2]$ is a $\left(\frac{|V_2|-|V_1|}{2}\right)$-regular graph and $\left|\mathcal{C}_{G[V_2]}\right|=|V_2|-|V_1|-1$.
\end{itemize}
\end{theorem}

\begin{theorem}\label{HP}
There exists $n_0\in \mathbb{N}$ such that if $G$ is a strongly edge-colored $n$-vertex graph with $n\geq n_0$ and $\delta(G)\geq \frac{n-1}{2}$, then  
\begin{itemize}[leftmargin=2em]
    \item either $G$ contains a rainbow Hamilton path, 
    \item or $G$ admits a partition $V(G)=V_1\cup V_2$ with $|V_1|\leq |V_2|$ such that $G[V_2]$ is a $\left(\frac{|V_2|-|V_1|-1}{2}\right)$-regular graph and $\left|\mathcal{C}_{G[V_2]}\right|=|V_2|-|V_1|-2$.
\end{itemize}
\end{theorem}

\begin{proof}[Proof of Theorem \ref{HP}]
Let $G$ be a strongly edge-colored graph on $n$ vertices with $\delta(G)\geq\frac{n-1}{2}$. Now we construct a new graph $H$ by joining a new vertex $v$ to all vertices of $G$ and painting these new edges with distinct colors not in $\mathcal{C}_G$. It is routine to verify that 
\begin{itemize}[leftmargin=2em]
  \item $\delta(H)\geq\frac{n+1}{2}$,
  \item $H$ is a strongly edge-colored graph, and
  \item $G$ has a rainbow Hamilton path if and only if $H$ has a rainbow Hamilton cycle.
\end{itemize}
By Theorem \ref{HC}, there exists $n_0\in \mathbb{N}$ such that if $n\geq n_0$, then either $H$ contains a rainbow Hamilton cycle, or $H$ admits a partition $V(H)=U_1\cup U_2$ with $|U_1|\leq |U_2|$ such that $H[U_2]$ is a $\left(\frac{|U_2|-|U_1|}{2}\right)$-regular graph and $\left|\mathcal{C}_{H[U_2]}\right|=|U_2|-|U_1|-1$. Therefore, if $H$ contains a rainbow Hamilton cycle, then $G$ has a rainbow Hamilton path. Otherwise, as $v$ is adjacent to all vertices of $G$ and $H[U_2]$ is a $\left(\frac{|U_2|-|U_1|}{2}\right)$-regular graph, we conclude that $v\in U_1$. So we get the desired partition of $V(G)$ by letting $V_1=U_1\setminus\{v\}$ and $V_2=U_2$.
\end{proof}

It is worth pointing out that combined with the following \Cref{regular}, our results actually characterize all the extremal graphs. To state \Cref{regular}, we need the following notation.
Given two graphs $F$ and $H$, a \textit{homomorphism} from $F$ to $H$ is a map $f: V(F)\rightarrow V(H)$ such that if $uv\in E(F)$, then $f(u)f(v)\in E(H)$. Moreover, we say that $F$ \textit{covers} $H$ if there is a surjective homomorphism $f$ from $F$ to $H$ such that for every vertex $v\in V(F)$ its neighborhood is bijectively mapped onto the neighborhood of $f(v)$. 
The \textit{Kneser graph} $K(k,s)$, with $k\geq 2s+1$ and $s\geq 2$, is the graph with vertex set $\binom{[k]}{s}$ such that two vertices are adjacent if and only if they are disjoint. The following result characterizes all the $k$-regular graphs with strong chromatic index $2k-1$.

\begin{theorem}[Lu\v{z}ar, M\'{a}\v{c}ajov\'{a},  \v{S}koviera and Sot\'{a}k, \cite{strongcolor}]\label{regular}
A $k$‐regular graph $G$ admits a strong edge coloring with $2k-1$ colors if and only if $G$ covers the Kneser graph $K(2k-1,k-1)$.    
\end{theorem}

At the end of this section, we use \Cref{HP} to settle another conjecture in strongly edge-colored graphs for sufficiently large $n$.
A Latin square of order $n$ is an $n\times n$ array, with each entry filled with an element from $[n]$, such that each element occurs exactly once in each row and in each column. A transversal of size $k$ in a Latin square is a set of $k$ elements, each of which comes from distinct rows and columns. The Ryser conjecture \cite{Ryser} states that every Latin square of order $n$ has a transversal of size $n$ if $n$ is odd. This conjecture is equivalent to that every properly edge-colored $K_{n,n}$ with $|\mathcal{C}_{K_{n,n}}|=n$ has a perfect rainbow matching if $n$ is odd. In \cite{wang}, Wang asked a more general question:  Is there a function $f(\delta)$ such that every properly edge-colored graph $G$ with $\delta(G)=\delta$ and $|V(G)|\geq f(\delta)$ contains a rainbow matching of size $\delta$? A result due to Lo \cite{lomat} states that $f(\delta)\leq 3.5\delta+2$, which is the best currently. Babu, Chandran and Vaidyanathan \cite{strongmat} first extended the above question to strongly edge-colored graphs. Later, Wang, Yan and Yu \cite{wang2} proved that every strongly edge-colored graph $G$ has a rainbow matching of size $\delta(G)$ if $|V(G)|\geq 2\delta(G)+2$, and conjectured the following.

\begin{conj}[Wang, Yan and Yu, \cite{wang2}]\label{conjmatching}
Let $G$ be a strongly edge-colored $n$-vertex graph with minimum degree $\delta$. If $n\in \{2\delta, 2\delta+1\}$, then $G$ has a rainbow matching of size $\delta$.
\end{conj}

In \cite{cheng3}, Cheng, Tan and Wang proved this conjecture for $n=2\delta+1$. As a straightforward application of Theorem \ref{HP}, we settle the remaining case of Conjecture \ref{conjmatching} for sufficiently large $n$.

\begin{theorem}\label{PM}
There exists $n_0\in \mathbb{N}$ such that for any even integer $n\geq n_0$, if $G$ is a strongly edge-colored $n$-vertex graph with $\delta(G)\geq \frac{n}{2}$, then $G$ contains a rainbow perfect matching.
\end{theorem}

\noindent\textbf{Notations.} Given a graph $G$ and two disjoint subsets $X,Y$ of $V(G)$, let $E_G(X,Y)=\{xy\in E(G) : x\in X, y\in Y\}$, and let $G[X,Y]$ be the bipartite graph with the vertex set $X\cup Y$ and the edge set $E_G(X,Y)$. For every $v\in V(G)$, we use $N_G(v,X)$ to denote the set of the neighbors of $v$ in $X$, and let $d_G(v,X)=|N_G(v,X)|$. 
In the following, we sometimes view an edge set as a graph.

\section{Preliminaries}

\subsection{Rainbow Hamilton cycles in $\mu n$-bounded graphs}

Strengthening the classical Dirac's theorem, Coulson and Perarnau \cite{bounded} proved the following theorem for $\mu n$-bounded graphs. 

\begin{theorem}[Coulson and Perarnau, \cite{bounded}]\label{bdHC}
There exists $\mu>0$ and $n_0\in \mathbb{N}$ such that if $n\geq n_0$ and $G$ is a $\mu n$-bounded $n$-vertex graph with $\delta(G)\geq \frac{n}{2}$, then $G$ contains a rainbow Hamilton cycle.
\end{theorem}

A key ingredient of the proof of Theorem \ref{bdHC} is the following switching Lemma (\Cref{switch}), which provides a sufficient condition for the existence of rainbow Hamilton cycles in $\mu n$-bounded graphs. To state Lemma \ref{switch}, we need the following notation. An \textit{oriented cycle} $\mathop{C}\limits^{\rightharpoonup}$ is a cycle with an orientation such that each vertex has an out-neighbor. We use $C$ to denote the underlying cycle of $\mathop{C}\limits^{\rightharpoonup}$. Given an oriented cycle $\mathop{C}\limits^{\rightharpoonup}$ and a vertex $v\in V(\mathop{C}\limits^{\rightharpoonup})$, the out-neighbor and the in-neighbor of $v$ are denoted by $v^+$ and $v^-$ respectively. For any edge $e=vv^+\in E(C)$ and any two vertices $u,w\in V(C)$ such that $e'=uw\notin E(C)$ and $v$ is in the oriented path from $w$ to $u$ induced by $\mathop{C}\limits^{\rightharpoonup}$, let $\mathop{C_1}\limits^{\rightharpoonup}=\mathcal{S}_1(\mathop{C}\limits^{\rightharpoonup},e,e')$ and $\mathop{C_2}\limits^{\rightharpoonup}=\mathcal{S}_2(\mathop{C}\limits^{\rightharpoonup},e,e')$ be the oriented cycles with $w=u^+$ such that the underlying cycles are, respectively,
$$C_1=\left(C-\{e,uu^+,w^-w\}\right)+\{e',vu^+,w^-v^+\};$$
$$C_2=\left(C-\{e,uu^+,w^-w\}\right)+\{e',vw^-,u^+v^+\}.$$
\noindent We call each of the above operations a \textit{switching}. Given a graph $G$, a switching is said to be \textit{admissible} if both of the underlying cycles are subgraphs of $G$. 

\begin{lemma}[Theorem 3.5, \cite{bounded}]\label{switch}
Let $n\in \mathbb{N}$, and suppose $1/n\ll \mu\ll\alpha\ll\beta\leq 1$. Let $G$ be a $\mu n$-bounded $n$-vertex graph, and let $Z\subseteq E(G)$ with $|Z|\leq\alpha n$ satisfying that each color in $Z$ is unique in $E(G)$. Suppose that $G$ has at least one Hamilton cycle containing $Z$. If for every oriented Hamilton cycle $\mathop{\mathcal{H}}\limits^{\rightharpoonup}$ of $G$ with $Z\subseteq E(\mathcal{H})$ and every edge $e\in E(\mathcal{H})\setminus Z$, there are at least $\beta n^2$ admissible switchings $\mathcal{S}_i(\mathop{\mathcal{H}}\limits^{\rightharpoonup};e,e')$ for some $e'\in E(G)\setminus E(\mathcal{H})$ and $i\in \{1,2\}$, then $G$ has a rainbow Hamilton cycle containing $Z$.
\end{lemma}

The next lemma from \cite{factor} is another important tool used to prove Theorem \ref{bdHC}, which provides a classification of graphs of high minimum degree.
Let $G$ be an $n$-vertex graph. For $0<\nu<1$ and $X\subseteq V(G)$, the \emph{$\nu$-robust neighborhood} of $X$ in $G$ is defined as $$RN_{\nu}(X):=\left\{v\in V(G) : |N_G(v)\cap X|\geq \nu n\right\}.$$ 
\noindent Let $0<\nu\leq \tau<1$. The graph $G$ is said to be a \emph{robust $(\nu,\tau)$-expander} if for every set $X\subseteq V(G)$ with $\tau n\leq |X|\leq (1-\tau)n$, it holds that $|RN_{\nu}(X)|\geq |X|+\nu n$.
Let $0<\gamma<1$. The graph $G$ is said to be \textit{$\gamma$-close} to the union of two disjoint copies of $K_{n/2}$, denoted by $2K_{n/2}$, if $\delta(G)\geq n/2$ and there exists $A\subseteq V(G)$ with $|A|=\lfloor\frac{n}{2}\rfloor$ such that there are at most $\gamma n^2$ edges between $A$ and $V(G)\setminus A$. And the graph $G$ is said to be \textit{$\gamma$-close} to $K_{n/2,n/2}$ if $\delta(G)\geq n/2$ and there exists $A\subseteq V(G)$ with $|A|=\lfloor\frac{n}{2}\rfloor$ such that there are at most $\gamma n^2$ edges with two ends in $A$.

\begin{lemma}[Lemma 1.3.2, \cite{factor}]\label{class}
Let $n\in \mathbb{N}$, and suppose that $0<1/n\ll \nu \ll \tau,\gamma<1$. Let $G$ be an $n$-vertex graph with $\delta(G)\geq n/2$. Then $G$ satisfies one of the following properties:
\begin{itemize}[leftmargin=2em]
  \item $G$ is $\gamma$-close to $2K_{n/2}$;
  \item $G$ is $\gamma$-close to $K_{n/2,n/2}$;
  \item $G$ is a robust $(\nu,\tau)$-expander.
\end{itemize}
\end{lemma}

\subsection{Strongly edge-colored graphs are almost $\mu n$-bounded}

The following so-called induced matching theorem was first proved by Ruzsa and Szemerédi \cite{IM} as a critical equivalent result of the famous $(6,3)$-Theorem by applying the Szemer\'edi Regularity Lemma, and can be used to derive a new proof of Roth's Theorem on arithmetic progressions. 

\begin{theorem}[Induced Matching Theorem, \cite{IM}]\label{IMT}
For any $\eta>0$, there exists $n_0\in \mathbb{N}$ such that if $G$ is a union of $n$ induced matchings and is on $n\geq n_0$ vertices, then $|E(G)|\leq \eta n^2$.
\end{theorem}

In our proof, a key observation is that the induced matching theorem implies that any strongly edge-colored graph is almost $\mu n$-bounded, which enables us to employ the tools and results from \cite{bounded} to prove Theorem \ref{HC} with some new ideas involved.

\begin{cor}\label{almostbd}
For any $0<\mu<1$, there exists $n_0\in \mathbb{N}$ such that for $n\geq n_0$, any strongly edge-colored $n$-vertex graph $G$ contains a $\mu n$-bounded subgraph $G'$ such that $d_{G'}(v)\geq d_G(v)-\mu^2n$ for every $v\in V(G)$. 
\end{cor}

\begin{proof}
Let $0<1/n\ll\eta\ll\mu<1$, and let $G$ be a strongly edge-colored $n$-vertex graph. Suppose that $E_1,E_2,\ldots, E_m$ are all the color classes of $G$ with $|E_1|\geq|E_2|\geq\cdots\geq|E_m|$. Since $G$ is a strongly edge-colored graph, each $E_i~(1\leq i\leq m)$ is an induced matching. Let $\ell=\max\{i : |E_i|\geq \mu n\}$. We claim that $\ell<\mu^2 n$. Otherwise, let $H=\bigcup_{i=1}^{\lceil\mu^2 n\rceil}E_i$. Then $|V(H)|\geq 2\mu n$ and Theorem \ref{IMT} implies that $|E(H)|\leq \eta |V(H)|^2\leq \eta n^2<\mu^3n^2$, which contradicts the fact that $|E(H)|=\sum_{i=1}^{\lceil\mu^2 n\rceil}|E_i|\geq \mu^3n^2 $. Therefore $G'=G-\left(\bigcup_{i=1}^{\ell}E_i\right)$ is as desired.  
\end{proof}

\section{Proof of Theorem \ref{HC}}

In this section, we prove Theorem~\ref{HC} by considering, in turn, robust expanders  (Lemma \ref{expander}), graphs close to $2K_{n/2}$ (Lemma \ref{twoclique}), and graphs close to $K_{n/2,n/2}$ (Lemma \ref{bip}). Theorem \ref{HC} then follows from Lemma \ref{class}.

\subsection{Robust expanders}
In this part, we prove that a strongly edge-colored robust expander of positive minimum degree contains a rainbow Hamilton cycle.
\begin{lemma}\label{expander}
Let $n\in \mathbb{N}$, and suppose $1/n\ll \nu\ll \tau\ll\gamma<1$. If $G$ is a strongly edge-colored $n$-vertex graph with $\delta(G)\geq\gamma n$ that is a robust $(\nu,\tau)$-expander, then $G$ contains a rainbow Hamilton cycle.
\end{lemma}

To prove Lemma \ref{expander}, we use the following theorem from \cite{bounded}.

\begin{theorem}[Theorem 7.1, \cite{bounded}]\label{bdexpander}
Let $n\in \mathbb{N}$, and suppose $1/n\ll\mu\ll \nu\ll \tau\ll\gamma<1$. If $G$ is a $\mu n$-bounded $n$-vertex graph with $\delta(G)\geq\gamma n$ that is a robust $(\nu,\tau)$-expander, then $G$ contains a rainbow Hamilton cycle.
\end{theorem}

\begin{proof}[Proof of Lemma \ref{expander}]
Set $1/n\ll\mu\ll \nu$. By Corollary \ref{almostbd}, there is a $\mu n$-bounded subgraph $G'$ of $G$ such that $d_{G'}(v)\geq d_G(v)-\mu^2n$ for every $v\in V(G)$. In particular, $G'$ is a spanning subgraph of $G$ with $\delta(G')\geq(\gamma-\mu^2)n>0$. Since $G$ is a robust $(\nu,\tau)$-expander, we know that $G'$ is a robust $(\nu-\mu^2,\tau)$-expander. Then by Theorem \ref{bdexpander}, $G'$ contains a rainbow Hamilton cycle.
\end{proof}

\subsection{Graphs close to $2K_{n/2}$}
This part is devoted to proving the following lemma, that is, Theorem \ref{HC} holds for graphs which are $\gamma$-close to $2K_{n/2}$.

\begin{lemma}\label{twoclique}
Let $n\in \mathbb{N}$, and suppose $1/n\ll\gamma\ll1$. If $G$ is a strongly edge-colored $n$-vertex graph with $\delta(G)\geq n/2$ that is $\gamma$-close to $2K_{n/2}$, then $G$ contains a rainbow Hamilton cycle.
\end{lemma}

Let $0\leq\varepsilon\leq1$. An $n$-vertex graph $G$ is an \emph{$\varepsilon$-superextremal two-clique} if there exists a partition $V(G)=V_1\cup V_2$ satisfying the following.
\stepcounter{propcounter}
\begin{enumerate}[label =\rm({\bfseries \Alph{propcounter}\arabic{enumi}})]
\item\label{A1} $\left||V_1|-|V_2|\right|\leq \varepsilon n$;
\item\label{A2} For $i\in\{1,2\}$, $d_G(v,V_i)\geq (1/4-\varepsilon)n$ for all vertices $v\in V_i$;
\item\label{A3} For $i\in\{1,2\}$, $d_G(v,V_i)\geq (1/2-\varepsilon)n$ for all but at most $\varepsilon n$ vertices $v\in V_i$. 
\end{enumerate}

The next lemma shows that any $n$-vertex graph which is close to $2K_{n/2}$ is also a superextremal two-clique. 

\begin{lemma}[Lemma 5.3, \cite{bounded}]\label{supertwoclique}
Let $n\in \mathbb{N}$, and suppose $1/n\ll\gamma\ll\varepsilon\ll 1$. Let $G$ be an $n$-vertex graph with $\delta(G)\geq n/2$ that is $\gamma$-close to $2K_{n/2}$. Then there is a partition $V(G)=V_1\cup V_2$ such that $G$ is an $\varepsilon$-superextremal two-clique with this partition. Moreover, if $|V_1|\leq|V_2|$, then either $\delta\left(G[V_1,V_2]\right)\geq 1$ or $d_G(v,V_2)\geq 2$ for every $v\in V_1$.
\end{lemma}

The following lemma shows that in a superextremal two-clique, we can always find a Hamilton cycle to include any two vertex disjoint edges between the two parts. 

\begin{lemma}[Lemma 5.5, \cite{bounded}]\label{HCtwoclique}
Let $G$ be an $\varepsilon$-superextremal two-clique with partition $V(G)=V_1\cup V_2$, and let $e_1, e_2$ be two vertex disjoint edges in $E_G(V_1,V_2)$. Then $G$ has a Hamilton cycle containing $e_1$ and $e_2$.
\end{lemma}

In order to apply Lemma \ref{switch} to find a rainbow Hamilton cycle, the last piece is to show that many admissible switchings exist in a superextremal two-clique.

\begin{lemma}[Lemma 5.6, \cite{bounded}]\label{switchtwoclique}
Let $n\in \mathbb{N}$, and suppose $1/n\ll\mu\ll\varepsilon\ll 1$. Let $G$ be an $\varepsilon$-superextremal two-clique with partition $V(G)=V_1\cup V_2$ satisfying $E_G(V_1,V_2)=Z$, where $Z$ is composed by two vertex disjoint edges between $V_1$ and $V_2$. Let $\mathop{\mathcal{H}}\limits^{\rightharpoonup}$ be an oriented Hamilton cycle of $G$. Then for every $e\in E(\mathcal{H})\setminus Z$, there are at least $n^2/300$ admissible switchings $\mathcal{S}_i(\mathop{\mathcal{H}}\limits^{\rightharpoonup};e,e')$ for some $e'\in E(G)\setminus E(\mathcal{H})$ and $i\in \{1,2\}$.
\end{lemma}

\begin{proof}[Proof of Lemma \ref{twoclique}]
Choose $0<1/n\ll\mu,\gamma\ll\varepsilon\ll1$. Let $G$ be a strongly edge-colored $n$-vertex graph with $\delta(G)\geq n/2$ which is $\gamma$-close to $2K_{n/2}$. Then by Lemma \ref{supertwoclique}, there is a partition $V_1\cup V_2$ of $V(G)$ with $|V_1|\leq |V_2|$ such that
\begin{itemize}[leftmargin=2em]
  \item $G$ is an $\varepsilon$-superextremal two-clique with this partition, and
  \item either $\delta\left(G[V_1,V_2]\right)\geq 1$ or $d_G(v,V_2)\geq 2$ for every $v\in V_1$.
\end{itemize}

In order to apply Lemmas~\ref{switchtwoclique} and~\ref{switch} to obtain a rainbow Hamilton cycle in $G$, 
we first construct a rainbow edge set $Z$ consisting of two vertex-disjoint edges between $V_1$ and $V_2$. 
By~\ref{A2}, for each $i\in\{1,2\}$ and every $v\in V_i$, we have
\[
d_G(v,V_i)\ge (1/4-\varepsilon)n \ge 2.
\]
If $d_G(v,V_2)\ge 2$ for every $v\in V_1$, then for any $v\in V_1$ and any $u\in N_G(v,V_1)$, 
we may choose two distinct vertices $w,w'\in V_2$ such that $vw,uw'\in E_G(V_1,V_2)$. 
Since the set $\{vw,uw'\}$ is not an induced matching, the colors of $vw$ and $uw'$ are distinct, 
and hence $Z=\{vw,uw'\}$ is the desired rainbow edge set.

Now suppose that $\delta(G[V_1,V_2])\ge 1$. 
Choose a vertex $v$ of maximum degree in $G[V_1,V_2]$, and assume without loss of generality that $v\in V_1$. 
Since $d_G(v,V_1)\ge 2$, we may choose a vertex $u\in N_G(v,V_1)$. 
Moreover, as $\delta(G[V_1,V_2])\ge 1$, there exist vertices $w,w'\in V_2$ such that 
$vw,uw'\in E_G(V_1,V_2)$, with $w\neq w'$ if possible.
If $w\neq w'$, then $Z=\{vw,uw'\}$ is the desired rainbow edge set. 
Otherwise, $w=w'$, and hence $N_G(v,V_2)=N_G(u,V_2)=\{w\}$, which implies that $d_G(w,V_1)\ge 2 > d_G(v,V_2)$,
contradicting the choice of $v$ as a vertex of maximum degree in $G[V_1,V_2]$.

Let $Z=\{e_1,e_2\}$ be the edge set constructed as above. By Corollary~\ref{almostbd}, 
$G$ contains a $\mu n$-bounded subgraph $G'$ such that 
$d_{G'}(v)\ge d_G(v)-\mu^2n$ for every $v\in V(G)$. 
Except for $e_1$ and $e_2$, we delete all edges of $G'[V_1,V_2]$ and all edges in the two color classes containing $e_1$ and $e_2$, respectively. 
Denote the resulting graph by $\hat{G}$. Recall that $\mu n\leq \varepsilon n$. Then $\hat{G}$ is a $\mu n$-bounded 
$2\varepsilon$-superextremal two-clique such that 
$E_{\hat{G}}(V_1,V_2)=Z$, and the color of each edge in $Z$ is unique in $E(\hat{G})$.
By Lemma~\ref{HCtwoclique}, $\hat{G}$ has a Hamilton cycle containing $e_1$ and $e_2$. 
By Lemma~\ref{switchtwoclique}, for any oriented Hamilton cycle 
$\mathop{\mathcal{H}}\limits^{\rightharpoonup}$ of $\hat{G}$ and every edge 
$e\in E(\mathcal{H})\setminus Z$, there are at least $n^2/300$ admissible switchings 
$\mathcal{S}_i(\mathop{\mathcal{H}}\limits^{\rightharpoonup};e,e')$ for some 
$e'\in E(\hat{G})\setminus E(\mathcal{H})$ and $i\in\{1,2\}$. 
Applying Lemma~\ref{switch}, we obtain a rainbow Hamilton cycle of $\hat{G}$, which is also a rainbow Hamilton cycle of $G$.
\end{proof}

\subsection{Graphs close to $K_{n/2,n/2}$}
In this part, we complete the proof of Theorem \ref{HC} by showing that Theorem \ref{HC} holds for graphs which are $\gamma$-close to $K_{n/2,n/2}$. 

\begin{lemma}\label{bip}
Let $n\in \mathbb{N}$, and suppose $1/n\ll\gamma\ll1$. Let $G$ be a strongly edge-colored $n$-vertex graph with $\delta(G)\geq n/2$. If $G$ is $\gamma$-close to $K_{n/2,n/2}$, then either $G$ contains a rainbow Hamilton cycle, or there is a partition $V(G)=V_1\cup V_2$ with $|V_1|\leq |V_2|$ such that $G[V_2]$ is a $\left(\frac{|V_2|-|V_1|}{2}\right)$-regular graph and $\left|\mathcal{C}_{G[V_2]}\right|=|V_2|-|V_1|-1$.
\end{lemma}

Let $0\leq\alpha, \varepsilon, \nu\leq 1$. An $n$-vertex graph $G$ is an \emph{$(\alpha,\varepsilon,\nu)$-superextremal biclique} if there exists a partition $V(G)=V_1\cup V_2$ satisfying the following.
\stepcounter{propcounter}
\begin{enumerate}[label =\rm({\bfseries \Alph{propcounter}\arabic{enumi}})]
\item\label{B1}  $0\leq|V_2|-|V_1|\leq \alpha n$;
  \item\label{B2} $d_G(v,V_2)\geq \nu n$ for all vertices $v\in V_1$;
  \item\label{B3} $d_G(v,V_1)\geq (1/4-\varepsilon)n$ for all vertices $v\in V_2$;
  \item\label{B4} $d_G(v,V_2)\leq 2\nu n$ for all vertices $v\in V_2$, unless $|V_1|=\lfloor n/2\rfloor$;
  \item\label{B5} For $i\in\{1,2\}$, $d_G(v,V_{3-i})\geq (1/2-\varepsilon)n$ for all but at most $\varepsilon n$ vertices $v\in V_i$.
\end{enumerate}

The next lemma shows that any $n$-vertex graph with $\delta(G)\geq n/2$ which is close to $K_{n/2,n/2}$ is also a superextremal biclique.

\begin{lemma}[Lemma 6.3, \cite{bounded}]\label{superbiclique}
Let $n\in \mathbb{N}$, and suppose $1/n\ll\gamma\ll\alpha\ll\varepsilon\ll\nu\ll 1$. Let $G$ be an $n$-vertex graph with $\delta(G)\geq n/2$ that is $\gamma$-close to $K_{n/2,n/2}$. Then $G$ is an $(\alpha,\varepsilon,\nu)$-superextremal biclique.
\end{lemma}

The next two lemmas serve a similar purpose as Lemmas \ref{HCtwoclique} and \ref{switchtwoclique}.

\begin{lemma}[Lemma 6.7, \cite{bounded}]\label{HCbiclique}
Let $n\in \mathbb{N}$, and suppose $1/n\ll\alpha\ll\nu\ll\eta\ll 1$. Let $G$ be an $(\alpha,\eta,\nu)$-superextremal biclique on $n$ vertices with partition $V(G)=V_1\cup V_2$ such that $|V_1|\leq|V_2|$, and let $M$ be a matching in $G[V_2]$ of size $|V_2|-|V_1|$. Then $G$ has a Hamilton cycle containing $M$.
\end{lemma}

\begin{lemma}[Theorem 6.8, \cite{bounded}]\label{switchbi}
Let $n\in \mathbb{N}$, and suppose $1/n\ll\mu\ll\alpha\ll\beta\ll\nu\ll\eta\ll 1$. Let $G$ be an $(\alpha,\eta,\nu)$-superextremal biclique on $n$ vertices with partition $V(G)=V_1\cup V_2$ such that $|V_1|\leq|V_2|$, and let $M$ be a matching in $G[V_2]$ of size at most $\alpha n$. Suppose that $G$ and $M$ satisfy
\begin{itemize}
   \item $E(G[V_1])=\emptyset$ and $E(G[V_2])=M$;
   \item $\max\{d_{G}(u,V_2),d_{G}(v,V_1)\}\geq(1/2-\eta)n$ for all $u\in V_1, v\in V_2$ with $uv\in E(G)$.
 \end{itemize}
 Then for every oriented Hamilton cycle $\mathop{\mathcal{H}}\limits^{\rightharpoonup}$ of $G$ and every edge $e\in E(\mathcal{H})\setminus M$, there are at least $\beta n^2$ admissible switchings $\mathcal{S}_i(\mathop{\mathcal{H}}\limits^{\rightharpoonup};e,e')$ for some $e'\in E(G)\setminus E(\mathcal{H})$ and $i\in\{1,2\}$.
\end{lemma}

The final lemma we need gives a characterization of strongly edge-colored graphs $G$ of bounded maximum degree that do not contain a rainbow matching of size $2\delta(G)$.

\begin{lemma}\label{matching}
Let $1/n\ll\varepsilon\ll 1$ and $k\in \mathbb{N}$ with $k\leq 2\varepsilon n$. Let $G$ be a strongly edge-colored $n$-vertex graph with $\delta(G)\geq \frac{k}{2}$ and $\Delta(G)\leq\varepsilon n$. If $G$ contains no rainbow matching of size $k$, then $|\mathcal{C}_G|=k-1$.
\end{lemma}
\begin{proof}
Since $G$ is a strongly edge-colored graph and $\delta(G)\geq\frac{k}{2}$, it immediately follows that $|\mathcal{C}_G|\geq k-1$. Suppose on the contrary that $|\mathcal{C}_G|\geq k$, we derive a contradiction by finding a vertex of degree greater than $\varepsilon n$, and thus complete the proof.
\vspace{0.5em}

\noindent \textbf{Case 1.} There is a color class of size at most $\frac{k}{2}$.
\vspace{0.5em}

Suppose $uv$ is an edge from this color class, and let $H$ be the graph obtained from $G$ by deleting vertices $u,v$ and this color class. Note that $d_H(w)\geq \frac{k}{2}$ for all but at most $2\Delta(G)+k$ $(\leq 4\varepsilon n)$ vertices $w\in V(H)$.

Let $M$ be a maximum rainbow matching of $H$. Then $|M|\leq k-2$, 
since otherwise $M\cup\{uv\}$ would be a rainbow matching of $G$ of size at least $k$.
Let $H':=H\setminus V(M)$. Then we have
$$|E(H')|\geq \frac{kn}{4}-2\varepsilon n-\frac{k}{2}-2(k-2)\varepsilon n\geq \left(\frac{1}{4}-2\varepsilon\right)kn.$$
Since $M$ is a maximum rainbow matching of $H$, we have $\mathcal{C}_{H'}\subseteq \mathcal{C}_M$.
Hence, by the pigeonhole principle and the fact that $H$ is strongly edge-colored, there exists a monochromatic induced matching in $H'$ of size at least $(\frac{1}{4}-2\varepsilon)n$.
As there are at most $4\varepsilon n$ vertices of degree less than $\frac{k}{2}$ in $H$, we can get a monochromatic induced matching $M'$ of size at least $(\frac{1}{4}-6\varepsilon)n$ in $H'$ such that $d_H(w)\geq \frac{k}{2}$ for each $w\in V(M')$. 
Consequently, for each edge in $M'$, the edge together with all edges incident to it in $H$ use at least $k-1$ distinct colors.
Since $|\mathcal{C}_{H'}|\leq |\mathcal{C}_M|=|M|\leq k-2$, for every edge in $M'$ there exists an edge $f$ between $V(M)$ and $V(H')$ in $H$ such that $c(f)\notin \mathcal{C}_M$, where $c(f)$ denotes the color of $f$.

Let $M_0\subseteq M$ be the set of edges that are incident to at least one edge $f$ with an end in $V(H')$ and $c(f)\notin \mathcal{C}_M$. 
Then, for any $e\in M_0$, there is no edge $e'\in E(H')$ with $c(e')=c(e)$; otherwise, replacing $e$ by $e'$ and adding the corresponding edge between $V(M)$ and $V(H')$ would yield a larger rainbow matching, contradicting the maximality of $M$.
It follows that all colors appearing in $H'$ avoid $\mathcal C_{M_0}$.
Moreover, each edge of $M'$ contributes at least $|M_0|$ edges in $E_H(V(M),V(M'))$ whose color lies in $\mathcal C_{M_0}$ or outside $\mathcal C_M$, and each edge of $M\setminus M_0$ contributes at most $|M_0|$ edges in $E_H(V(M),V(M'))$ whose color lies in $\mathcal C_{M_0}$.
Thus, $|E_H(V(M_0),V(M'))|\geq (\frac{1}{4}-6\varepsilon)n|M_0|-k|M_0|$.
Therefore, there is a vertex $w\in V(M_0)$ satisfying
$$d_H(w)\geq \frac{(\frac{1}{4}-6\varepsilon)n|M_0|-k|M_0|}{2|M_0|}=\frac{n}{8}-3\varepsilon n-\frac{k}{2}\geq\frac{n}{8}-4\varepsilon n>\varepsilon n,$$
a contradiction.
\vspace{0.5em}

\noindent \textbf{Case 2.} Each color class is of size at least $\frac{k}{2}$.
\vspace{0.5em}

Let $M$ be a maximum rainbow matching of $G$. Since $|\mathcal{C}_G|\geq k$ and $|M|\leq k-1$, there is a color $c\in \mathcal{C}_G\setminus \mathcal{C}_M$. Let $H$ be the graph obtained from $G$ by deleting all vertices in $V(M)$. Then in $H$, there is no edge colored with $c$. On the other hand, in $G[V(M)]$, there are at most $\frac{k-1}{2}$ edges colored with $c$ (as each color class is an induced matching in $G$).  Therefore, there is at least one edge colored with $c$ in $G[V(M),V(H)]$.

Let $M_0\subseteq M$ be the set of edges that are incident to at least one edge $f$ with an end in $V(H)$ and $c(f)\in \mathcal{C}_G\setminus \mathcal{C}_M$. 
Then by a similar argument for \textbf{Case 1}, we can find a vertex $v$ of degree greater than $\varepsilon n$, which is a contradiction.
\end{proof}


\begin{proof}[Proof of Lemma \ref{bip}]

Choose $1/n\ll\mu,\gamma\ll\alpha\ll\beta,\varepsilon\ll\nu\ll\eta\ll1$. Let $G$ be a strongly edge-colored graph on $n$ vertices with $\delta(G)\geq n/2$ which is $\gamma$-close to $K_{n/2,n/2}$. By Lemma \ref{superbiclique}, there is a partition $V(G)=V_1\cup V_2$ with $|V_1|\leq|V_2|$ such that $G$ is an $(\alpha,\varepsilon,\nu)$-superextremal biclique with this partition.

Suppose first that $G[V_2]$ contains a rainbow matching $M$ of size $|V_2|-|V_1|$. By \cref{almostbd}, there is a $\mu n$-bounded subgraph $G'$ of $G$ such that $d_{G'}(v)\geq d_{G}(v)-\mu^2 n$ for every vertex $v\in V(G)$. Let $G''$ be the graph obtained from $G'$ by deleting all the edges $uv$ satisfying one of the following:
\begin{itemize}
  \item $c(uv)\in\mathcal{C}_M$,
  \item $uv\in G[V_1]\cup G[V_2]$,
  \item $u\in V_1, v\in V_2$ and $\max\{d_{G'}(u,V_2),d_{G'}(v,V_1)\}<(1/2-2\varepsilon)n$.
\end{itemize}
Let $\hat{G}=G''\cup M$. It is routine to check that $\hat{G}$ is an $(\alpha,\eta,\nu/2)$-superextremal biclique such that
\begin{itemize}
  \item $E(\hat{G}[V_1])=\emptyset$ and $E(\hat{G}[V_2])=M$,
  \item $\max\{d_{\hat{G}}(u,V_2),d_{\hat{G}}(v,V_1)\}\geq(1/2-\eta)n$ for all $u\in V_1, v\in V_2$ with $uv\in E(\hat{G})$,
  \item each edge in $M$ has a unique color in $E(\hat{G})$,
  \item the strong edge coloring of $\hat{G}$ is $\mu n$-bounded.
\end{itemize}
Thus, by Lemma \ref{HCbiclique}, $\hat{G}$ has a Hamilton cycle containing $M$. By Lemma \ref{switchbi}, for every oriented Hamilton cycle $\mathop{\mathcal{H}}\limits^{\rightharpoonup}$ of $\hat{G}$ and every edge $e\in E(\mathcal{H})\setminus M$, there are at least $\beta n^2$ admissible switchings $\mathcal{S}_i(\mathop{\mathcal{H}}\limits^{\rightharpoonup};e,e')$ for some $e'\in E(\hat{G})\setminus E(\mathcal{H})$ and $i\in\{1,2\}$. Then by Lemma \ref{switch}, $\hat{G}$ contains a rainbow Hamilton cycle, which is also a rainbow Hamilton cycle of $G$.

If $G[V_2]$ contains no rainbow matching of size $|V_2|-|V_1|$, then $|\mathcal{C}_{G[V_2]}|=|V_2|-|V_1|-1$ follows from Lemma \ref{matching} and the fact that $\delta(G[V_2])\geq \frac{n}{2}-|V_1|=\frac{|V_2|-|V_1|}{2}$. Therefore, $G[V_2]$ must be a $\left(\frac{|V_2|-|V_1|}{2}\right)$-regular graph, since $G[V_2]$ is strongly edge-colored.
\end{proof}

\section{Concluding remarks}

In this paper, by establishing a connection between strongly edge-colored graphs and $\mu n$-bounded graphs, 
we determine the optimal minimum degree condition guaranteeing the existence of rainbow Hamilton cycles in strongly edge-colored graphs. 
In recent years, Dirac-type problems concerning rainbow pancyclicity of strongly edge-colored graphs have also been extensively studied \cite{panconnect,edgepan,pairpan,vertexpan}. 
We believe that our approach shows promising potential for addressing these problems.

\bibliography{reference.bib}

\end{document}